\documentclass[12pt,oneside,reqno]{amsart}

\usepackage{amsmath,amssymb,amsthm}
\usepackage[margin=1.1in]{geometry}
\usepackage[hidelinks]{hyperref}
\hypersetup{
pdftitle={Fourier Ratios of Graph Kernels: Energy Bounds, Optimal Labelings, and Recovery},
pdfauthor={Vishal Gupta and Alex Iosevich},
pdfkeywords={Fourier ratio, graph energy, graph layout, Fourier algebra, circulant graph, spectral projector, compressed sensing}
}

\numberwithin{equation}{section}

\newtheorem{theorem}{Theorem}[section]
\newtheorem{proposition}[theorem]{Proposition}
\newtheorem{corollary}[theorem]{Corollary}
\newtheorem{lemma}[theorem]{Lemma}
\newtheorem{question}[theorem]{Question}

\theoremstyle{definition}

\theoremstyle{remark}
\newtheorem{remark}[theorem]{Remark}

\begin{document}

\title[Fourier Ratios of Graph Kernels]{Fourier Ratios of Graph Kernels: Energy Bounds, Optimal Labelings, and Recovery}

\author{Vishal Gupta}
\thanks{Department of Mathematics, University of Rochester, Rochester, New York 14627, USA. Email addresses: vishalgupta@rochester.edu and alex.iosevich@rochester.edu.}

\author{Alex Iosevich}
\thanks{Corresponding author: Alex Iosevich.}

\date{}


\begin{abstract}
We study labeling-sensitive Fourier complexity for finite graph kernels. After identifying the vertices of a graph with the cyclic group $\mathbb Z_N$, its adjacency matrix becomes a function on $\mathbb Z_N^2$. Minimizing the quotient of the $\ell^1$ and $\ell^2$ norms of its two-dimensional Fourier transform over all vertex labelings gives an isomorphism invariant $\operatorname{FR}_{\min}(G)$.

A nuclear-norm argument gives
\[
\operatorname{FR}_{\min}(G)
\geq
\frac{\mathcal E(G)}{\sqrt{2s}},
\]
where $s$ is the number of edges and $\mathcal E(G)$ is the graph energy. The natural cyclic labeling attains equality for every circulant graph. We obtain exact formulas for several graph families and a labeling-sensitive complete bipartite example. We also connect the invariant with the Fourier algebra of $\mathbb Z_N^2$. The quantitative Cohen idempotent theorem implies that every Boolean kernel of bounded Fourier ratio has an exact signed coset decomposition whose length is independent of $N$.

A previously established Fourier-ratio recovery theorem gives stable Frobenius approximation of a fixed labeled adjacency matrix from Bernoulli samples. We distinguish this conclusion from exact edge recovery and from the problem of finding a good labeling.

For a Laplacian eigenvalue of multiplicity $m(\lambda)$, we prove
\[
\operatorname{FR}_{\min}(\Pi_\lambda)
\geq
\sqrt{m(\lambda)},
\]
with equality for circulant graphs. Strongly regular graphs and the Petersen graph show how adjacency and projector complexity can agree or differ. Direct projector sampling yields heat-kernel approximation. We conclude with a graph-signal spectral synthesis principle and asymptotic uniqueness from incomplete vertex data.
\end{abstract}

\maketitle

\section{Introduction}

Recovering a graph from partial information is a central problem in network theory, data science, signal processing, and machine learning. In many situations, one does not observe the full adjacency matrix, but only a random subset of its entries. The purpose of this paper is to propose a Fourier-analytic framework in which the relevant complexity is determined by the additive spectral organization of the edge relation.

The framework considered here is different from several familiar models. Low-rank matrix completion is based on rank and incoherence assumptions \cite{CandesRecht}. Graphon estimation is based on probabilistic latent-variable structure \cite{GaoLuZhou}. Graph-signal sampling usually concerns a signal defined on a graph that is already known \cite{ChenEtAl}. In the present paper, the principal object is a deterministic graph kernel, and the structural hypothesis is that the edge relation becomes Fourier-compressible after the vertices have been placed in cyclic coordinates.

Let $G=(V,E)$ be a nonempty simple graph with $|V|=N$. A labeling is a bijection
\[
\sigma:\mathbb Z_N\longrightarrow V.
\]
Under this labeling, the adjacency matrix becomes the function
\[
A_\sigma(x,y)=1_{\{\sigma(x)\sim\sigma(y)\}},
\]
defined on $\mathbb Z_N^2$. The two-dimensional Fourier transform is
\[
\widehat F(m,n)
=
\frac{1}{N}
\sum_{x,y\in\mathbb Z_N}
F(x,y)e^{-\frac{2\pi i(mx+ny)}{N}}.
\]
With this normalization, Parseval's identity takes the form
\[
\|\widehat F\|_2=\|F\|_2.
\]
For $F\neq0$, define
\[
\operatorname{FR}(F)
=
\frac{\|\widehat F\|_1}{\|\widehat F\|_2}.
\]
The main graph invariant studied in this paper is
\[
\operatorname{FR}_{\min}(G)
=
\min_{\sigma:\mathbb Z_N\to V}
\operatorname{FR}(A_\sigma).
\]

From this point of view, graph recovery becomes a hidden structure problem. One asks whether the vertices can be organized in additive coordinates for which the edge relation has a concentrated Fourier representation. Recoverability is therefore not determined only by the number of edges, the degree distribution, or the rank of the adjacency matrix. A dense graph may have a very small Fourier ratio when its edge relation is generated by a simple additive rule, while a sparse graph may have a much larger Fourier ratio when its edges are poorly aligned with the cyclic coordinates.

This distinction can already be seen in complete bipartite graphs. If the two parts of $K_{\frac{N}{2},\frac{N}{2}}$ are labeled by the even and odd residue classes, then the edge relation is determined by one character and its Fourier ratio equals $\sqrt{2}$. If the two parts are labeled by consecutive intervals, then the same abstract graph has Fourier ratio of order $(\log N)^2$. Thus, low Fourier complexity is not the same as edge sparsity, and it is not even a property of a displayed adjacency matrix without reference to the coordinates used to display it.

The dependence on labeling places the problem in the general setting of graph layout. Classical parameters such as bandwidth, minimum linear arrangement, cutwidth, and seriation seek orderings that reveal local or geometric regularity \cite{AtkinsBomanHendrickson,DiazPetitSerna}. A small Fourier ratio favors a different kind of organization. It favors global additive patterns in $\mathbb Z_N^2$, including diagonal translates, congruence classes, and affine relations. The parity labeling of a complete bipartite graph is a useful example because it is not geometrically local, but it is extremely simple from the additive Fourier point of view.

The central structural result of the paper is a lower bound in terms of graph energy. If $s=|E|$ and
\[
\mathcal E(G)
=
\sum_{j=1}^N|\lambda_j(A)|,
\]
then
\[
\operatorname{FR}_{\min}(G)
\geq
\frac{\mathcal E(G)}{\sqrt{2s}}.
\]
The proof is based on the nuclear norm. The two-dimensional Fourier transform multiplies the adjacency matrix on the left and on the right by unitary Fourier matrices, so it preserves the singular values. The entrywise $\ell^1$ norm dominates the nuclear norm, while the Frobenius norm of the adjacency matrix equals $\sqrt{2s}$.

For a circulant graph, the natural cyclic labeling makes the transformed adjacency matrix anti-diagonal, and the entries on the anti-diagonal are exactly the adjacency eigenvalues. The graph-energy lower bound is therefore attained. In particular, the natural labeling is globally optimal for every circulant graph. This observation gives exact formulas for complete graphs, balanced Tur\'an graphs, cycles, and parity-labeled complete bipartite graphs.

The invariant also has a finite Fourier-algebra interpretation. If $F$ is a Boolean kernel on $\mathbb Z_N^2$ with density $\alpha$, then its Fourier algebra norm equals
\[
\sqrt\alpha\operatorname{FR}(F).
\]
The quantitative Cohen idempotent theorem therefore gives an exact signed decomposition of every bounded-ratio Boolean kernel into a number of cosets of subgroups controlled only by its Fourier ratio and density. This supplies an inverse theorem in which low Fourier ratio forces additive structure. The decomposition may involve cancellation, and the symmetry and zero diagonal of a graph kernel create further questions that are not answered by the general theorem.

The recovery theorem used in this paper was developed in the Fourier-ratio literature \cite{AldalehEtAl,BursteinIosevichNathan}. Once a labeling with small Fourier ratio is available, the labeled adjacency matrix can be approximated stably in Frobenius norm from random sampled entries by minimizing the $\ell^1$ norm of its Fourier transform. Two qualifications are important. The theorem assumes that the labeling is already known, and it does not provide an algorithm for finding a labeling that realizes $\operatorname{FR}_{\min}(G)$. The conclusion is also a relative Frobenius-error estimate. It is not, by itself, an exact edge-recovery theorem.

A second theme is that Fourier-compressibility may arise at the level of the harmonic decomposition of the graph Laplacian. For a Laplacian eigenvalue $\lambda$, let $\Pi_\lambda$ denote the orthogonal projector onto the corresponding eigenspace. After choosing a labeling, the kernel $\Pi_\lambda(x,y)$ becomes a function on $\mathbb Z_N^2$. We define its minimum Fourier ratio over all labelings and prove
\[
\operatorname{FR}_{\min}(\Pi_\lambda)
\geq
\sqrt{m(\lambda)},
\]
where $m(\lambda)$ is the multiplicity of $\lambda$. For circulant graphs, the natural cyclic labeling attains equality for every eigenspace.

Edge complexity and harmonic complexity are related, but they are not the same. The complete graph has uniformly bounded adjacency Fourier ratio, while its nontrivial Laplacian projector has Fourier ratio $\sqrt{N-1}$. The cycle graph has adjacency Fourier ratio of order $\sqrt N$, while every Laplacian projector has Fourier ratio at most $\sqrt{2}$. For strongly regular graphs, the projector matrices are affine combinations of the adjacency matrix, the identity matrix, and the all-ones matrix, and this gives explicit inequalities between the two kinds of complexity. The Petersen graph shows that a labeling minimizing the adjacency Fourier ratio need not minimize either nontrivial projector Fourier ratio.

The adjacency and projector recovery statements use different observation models. Adjacency recovery begins with sampled entries of the adjacency matrix. Projector recovery begins with sampled entries of the projector kernel itself. Partial adjacency data does not directly provide partial projector data. We therefore treat projector recovery as a separate kernel-sampling problem and derive heat-kernel approximation only within that model.

The final part of the paper concerns a related but distinct question. Instead of a kernel on $\mathbb Z_N^2$, we consider a graph signal on $\mathbb Z_N$ and measure the distribution of its norm among Laplacian eigenspaces. A spectral-block Fourier ratio gives a quantitative restriction on how such a signal can concentrate. After the vertex-space norms are normalized explicitly, the resulting argument yields a graph analogue of spectral synthesis: a spectrally regular signal supported on a sublinear set must have normalized $\ell^1$ norm tending to zero. This does not follow from projector-kernel recovery, and it does not imply exact uniqueness at a fixed value of $N$. It is an asymptotic statement about graph signals.

The paper is organized as follows. Section \ref{sec:preliminaries} introduces the Fourier ratio, the labeling invariants, and the recovery theorem. Section \ref{sec:layout} discusses labeling dependence, graph layout, random relabeling, and comparison with classical spectral quantities. Section \ref{sec:energy} proves the universal Fourier lower bound and the graph-energy lower bound. Section \ref{sec:circulant} studies circulant graphs and exact examples. Section \ref{sec:extensions} treats affine kernels, Fourier algebra structure, weighted and directed matrices, perturbations, and the entropy interpretation. Section \ref{sec:projectors} develops the spectral-projector theory. Section \ref{sec:srg} studies strongly regular graphs and the Petersen graph. Section \ref{sec:projector-recovery} gives the projector-sampling and heat-kernel statements. Section \ref{sec:synthesis} proves the graph-signal spectral synthesis theorem. Section \ref{sec:questions} contains open problems. Appendix \ref{app:petersen-code} contains complete code for the Petersen computation.

\section{Fourier ratios and the recovery theorem}\label{sec:preliminaries}

Throughout the sections concerning kernels on $\mathbb Z_N^2$, all norms use counting measure. For $F:\mathbb Z_N^2\to\mathbb C$ and $1\leq q<\infty$, define
\[
\|F\|_q
=
\left(\sum_{x,y\in\mathbb Z_N}|F(x,y)|^q
\right)^{\frac{1}{q}}.
\]
The Fourier transform is
\begin{equation}\label{eq:fourier-transform}
\widehat F(m,n)
=
\frac{1}{N}
\sum_{x,y\in\mathbb Z_N}
F(x,y)e^{-\frac{2\pi i(mx+ny)}{N}}.
\end{equation}
The inversion formula is
\[
F(x,y)
=
\frac{1}{N}
\sum_{m,n\in\mathbb Z_N}
\widehat F(m,n)e^{\frac{2\pi i(mx+ny)}{N}},
\]
and Parseval's identity is
\[
\|\widehat F\|_2=\|F\|_2.
\]
For $F\neq0$, set
\begin{equation}\label{eq:fourier-ratio}
\operatorname{FR}(F)
=
\frac{\|\widehat F\|_1}{\|\widehat F\|_2}.
\end{equation}
Since $\widehat F$ has $N^2$ coordinates, Cauchy--Schwarz gives
\begin{equation}\label{eq:basic-fr-bounds}
1\leq\operatorname{FR}(F)\leq N.
\end{equation}
If $\widehat F$ is supported on at most $M$ frequencies, then
\begin{equation}\label{eq:support-upper-bound}
\operatorname{FR}(F)\leq\sqrt M.
\end{equation}

For a simple graph $G=(V,E)$ and a labeling $\sigma:\mathbb Z_N\to V$, let $A_\sigma$ denote the associated adjacency kernel. If $s=|E|$, then
\begin{equation}\label{eq:adjacency-frobenius}
\|A_\sigma\|_2^2=2s.
\end{equation}
We define
\begin{equation}\label{eq:fr-min-definition}
\operatorname{FR}_{\min}(G)
=
\min_{\sigma:\mathbb Z_N\to V}
\operatorname{FR}(A_\sigma)
\end{equation}
and
\begin{equation}\label{eq:fr-max-definition}
\operatorname{FR}_{\max}(G)
=
\max_{\sigma:\mathbb Z_N\to V}
\operatorname{FR}(A_\sigma).
\end{equation}
Both quantities are isomorphism invariants. The quotient
\[
\frac{\operatorname{FR}_{\max}(G)}{\operatorname{FR}_{\min}(G)}
\]
measures the instability of the displayed Fourier complexity under relabeling.

We record the finite-group recovery theorem in the form needed below. The result first appears in \cite[Theorem 1.21]{AldalehEtAl} and is restated in \cite[Theorem 3]{BursteinIosevichNathan}. For $X\subset H$, write
\[
\|h\|_{\ell^2(X)}
=
\left(\sum_{x\in X}|h(x)|^2
\right)^{\frac{1}{2}}.
\]

\begin{theorem}[Fourier-ratio recovery]\label{thm:fr-recovery}
Let $H$ be a finite abelian group of cardinality $M$, and let $f:H\to\mathbb C$ be nonzero. Suppose that the unitary Fourier transform of $f$ satisfies
\[
\operatorname{FR}(f)
=
\frac{\|\widehat f\|_1}{\|\widehat f\|_2}
\leq r.
\]
Fix $\epsilon\in(0,1)$, and form a random set $X\subset H$ by retaining each point independently with probability $p$. There are absolute constants $c,C>0$ such that, if
\begin{equation}\label{eq:recovery-sample-bound}
pM
\geq
C
\frac{r^2}{\epsilon^2}
\log^2(\frac{r}{\epsilon})
\log M,
\end{equation}
then, with probability at least $1-e^{-cpM}$, a minimizer $f^\sharp$ of
\begin{equation}\label{eq:recovery-program}
\min_g\|\widehat g\|_1
\text{subject to}
\|g-f\|_{\ell^2(X)}
\leq
\epsilon\|f\|_2
\end{equation}
satisfies
\begin{equation}\label{eq:recovery-error}
\|f^\sharp-f\|_2
\leq
11.47\epsilon\|f\|_2.
\end{equation}
\end{theorem}

Applying Theorem \ref{thm:fr-recovery} to $H=\mathbb Z_N^2$ gives the graph-kernel consequence.

\begin{corollary}[Stable approximation of a labeled adjacency matrix]\label{cor:adjacency-recovery}
Let $G$ be a simple graph with $N$ vertices and $s$ edges, and fix a labeling $\sigma:\mathbb Z_N\to V(G)$. Suppose
\[
\operatorname{FR}(A_\sigma)\leq r.
\]
Let $X\subset\mathbb Z_N^2$ be a Bernoulli-$p$ sample. If
\begin{equation}\label{eq:graph-sample-bound}
pN^2
\geq C\frac{r^2}{\epsilon^2}
\log^2\left(\frac{r}{\epsilon}\right)
\log(N^2),
\end{equation}
then, with probability at least $1-e^{-cpN^2}$, a minimizer $A^\sharp$ of \eqref{eq:recovery-program}, with $f=A_\sigma$, satisfies
\begin{equation}\label{eq:adjacency-error}
\|A^\sharp-A_\sigma\|_2
\leq
11.47\epsilon\sqrt{2s}.
\end{equation}
\end{corollary}

\begin{remark}\label{rem:recovery-scope}
Corollary \ref{cor:adjacency-recovery} is a stable Frobenius-norm approximation theorem for a fixed labeled matrix. It does not produce a labeling with small Fourier ratio, and it does not assert exact recovery of the binary edge set at fixed $\epsilon$.

For example, define the entrywise thresholded matrix
\[
\overline A(x,y)
=
1_{\{\operatorname{Re}A^\sharp(x,y)\geq\frac{1}{2}\}}.
\]
Every ordered entry on which $\overline A$ and $A_\sigma$ differ contributes at least $\frac{1}{4}$ to $\|A^\sharp-A_\sigma\|_2^2$. Hence
\begin{equation}\label{eq:hamming-bound}
\#\{(x,y):\overline A(x,y)\neq A_\sigma(x,y)\}
\leq
8(11.47)^2\epsilon^2s.
\end{equation}
A sufficient condition for exact recovery after thresholding is
\[
11.47\epsilon\sqrt{2s}<\frac{1}{2}.
\]
Thus, the accuracy parameter would have to decrease with the size of the graph. Likewise, if the right side of \eqref{eq:graph-sample-bound} exceeds $N^2$, the theorem gives no subquadratic sampling guarantee. This does not constitute an impossibility result.

The feasible radius in \eqref{eq:recovery-program} contains $\|A_\sigma\|_2=\sqrt{2s}$. In applications, one must therefore assume that $s$ is known or replace this quantity by a justified estimate. We do not address the estimation of this radius here.
\end{remark}

The role of the Fourier ratio in the sample bound is easiest to see by suppressing the common accuracy and logarithmic factors. A bounded Fourier ratio gives a polylogarithmic sufficient sample count. A Fourier ratio of order $\sqrt N$ gives a nearly linear sufficient sample count. A Fourier ratio of order $N$ makes the sufficient condition comparable with the size of the entire matrix. These are consequences of one recovery theorem, not information-theoretic lower bounds.

\section{Labeling dependence and graph layout}\label{sec:layout}

The abstract graph does not come with a preferred identification of its vertices with $\mathbb Z_N$. The quantity $\operatorname{FR}_{\min}(G)$ asks for the best additive coordinate system, while $\operatorname{FR}_{\max}(G)$ records the worst one. From the point of view of recovery, the distinction is essential because Theorem \ref{thm:fr-recovery} applies to the displayed kernel in the observed coordinates.

A labeling with small Fourier ratio should not be confused with a labeling that places most edges near the diagonal. If an edge set is described by a congruence condition, a diagonal translate, or a small collection of affine relations, then its Fourier transform is concentrated even when the displayed edges are spread throughout the matrix. By contrast, a labeling that produces a visually narrow band may still have a diffuse two-dimensional Fourier transform.

This gives a Fourier-analytic graph layout problem. Given $G$, one would like to find a permutation matrix $P$ for which the matrix $P^TAP$, viewed as a function on $\mathbb Z_N^2$, has a small Fourier ratio. The objective is global and nonlinear. It is also different from a spectral invariant of $A$, since conjugation by $P$ preserves all adjacency and Laplacian eigenvalues but may substantially alter the two-dimensional Fourier transform.

\subsection{A heuristic for random relabeling}

The cycle gives a useful model. Under its natural labeling, the edge set lies on the two relations $y=x+1$ and $y=x-1$, and the Fourier transform is supported on one anti-diagonal. If the vertices are relabeled by a uniformly random permutation, the same $2N$ ordered edges are distributed through $\mathbb Z_N^2$ in a highly dependent but apparently unstructured way.

For a fixed nonzero frequency $(a,b)$, one may write
\[
\widehat{A_\sigma}(a,b)
=
\frac{1}{N}
\sum_{(x,y)\in E_\sigma^{\operatorname{ord}}}
e^{-\frac{2\pi i(ax+by)}{N}},
\]
where $E_\sigma^{\operatorname{ord}}$ denotes the ordered edge set. A random-phase model treats the summands as if they were weakly dependent unit complex numbers. If there are $M$ ordered edges, this model predicts
\[
\mathbb E|\widehat{A_\sigma}(a,b)|
\approx
\frac{\sqrt\pi}{2}
\frac{\sqrt M}{N}.
\]
Summing over approximately $N^2$ frequencies and dividing by $\|A_\sigma\|_2=\sqrt M$ gives the heuristic
\begin{equation}\label{eq:random-labeling-heuristic}
\operatorname{FR}(A_\sigma)
\approx
\frac{\sqrt\pi}{2}N.
\end{equation}

This calculation is not a theorem. The edge locations arising from a random permutation are strongly dependent, and the zero frequency has a different distribution from the other frequencies. Nevertheless, it gives a natural explanation for why a generic labeling may be much less compressible than a structured one. For a cycle, the natural value is of order $\sqrt N$, while the random-phase prediction is of order $N$.

The recovery interpretation must also be stated carefully. If the heuristic value in \eqref{eq:random-labeling-heuristic} is inserted into Corollary \ref{cor:adjacency-recovery}, the sufficient sample condition is no longer subquadratic after logarithmic factors are included. The theorem then becomes ineffective as a sparse-recovery result. It does not prove that recovery from fewer samples is impossible.

A rigorous analysis of random relabeling would require concentration inequalities for permutation statistics or a martingale argument that respects the dependencies among the edge locations. This appears to be a natural problem in its own right.

\subsection{Comparison with classical spectral quantities}

The Fourier ratio of a displayed adjacency matrix and the usual adjacency spectrum measure different properties. The adjacency spectrum is unchanged by relabeling. The Fourier ratio of the displayed kernel is not. The parity and interval labelings of $K_{\frac{N}{2},\frac{N}{2}}$ will give an exact example in which the spectrum is fixed while the Fourier ratio changes by an unbounded factor.

There is nevertheless one elementary spectral estimate. The zero Fourier coefficient is
\[
\widehat A_\sigma(0,0)
=
\frac{2s}{N}.
\]
Since $\|A_\sigma\|_2=\sqrt{2s}$,
\begin{equation}\label{eq:average-degree-bound}
\operatorname{FR}(A_\sigma)
\geq
\frac{\sqrt{2s}}{N}
=
\sqrt{\frac{\overline d}{N}},
\end{equation}
where $\overline d=\frac{2s}{N}$ is the average degree. This estimate is usually weak. The graph-energy bound proved in Section \ref{sec:energy} is substantially stronger.

The complete graph and the cycle show that there is no useful monotone comparison with the spectral radius or the Laplacian gap alone. For $K_N$, the adjacency spectral radius is $N-1$ and the nonzero Laplacian eigenvalue is $N$, while
\[
\operatorname{FR}_{\min}(K_N)
=
2\sqrt{1-\frac{1}{N}}.
\]
For $C_N$, the adjacency spectral radius is $2$ and the Laplacian gap is
\[
2-2\cos \left(\frac{2\pi}{N}\right),
\]
which is of order $N^{-2}$, while
\[
\operatorname{FR}_{\min}(C_N)
\sim
\frac{2\sqrt2}{\pi}\sqrt N.
\]
Thus, a graph may have a very large spectral radius and a bounded minimum Fourier ratio, or a bounded spectral radius and an unbounded minimum Fourier ratio.

The correct conclusion is not that Fourier complexity is unrelated to spectral graph theory. The graph-energy theorem gives a direct relation with the full adjacency spectrum. Rather, the Fourier ratio records how spectral data interacts with an external additive coordinate system, while the spectral radius and Laplacian gap are intrinsic quantities that do not see that coordinate system.

\section{Universal lower bounds and graph energy}\label{sec:energy}

We begin with a lower bound that uses only the fact that a binary kernel is loop-free. It does not require symmetry.

\begin{proposition}
\label{prop:loop-free-bound}
Let $F:\mathbb Z_N^2\to\{0,1\}$ be nonzero and satisfy $F(x,x)=0$ for every $x$. Let
\[
M
=
\sum_{x,y\in\mathbb Z_N}F(x,y).
\]
Then
\begin{equation}\label{eq:loop-free-max}
\operatorname{FR}(F)
\geq
\max\{
\frac{N}{\sqrt M},
\frac{2\sqrt M}{N}
\}
\geq
\sqrt2.
\end{equation}
\end{proposition}

\begin{proof}
Every Fourier coefficient satisfies
\[
|\widehat F(m,n)|
\leq
\frac{M}{N}.
\]
Since $\|\widehat F\|_2^2=M$,
\[
\|\widehat F\|_2^2
\leq
\|\widehat F\|_\infty\|\widehat F\|_1
\]
gives $\|\widehat F\|_1\geq N$. Hence
\[
\operatorname{FR}(F)
\geq
\frac{N}{\sqrt M}.
\]

The diagonal condition gives a second estimate. Fourier inversion and summation over $x$ give
\[
0
=
\sum_{x\in\mathbb Z_N}F(x,x)
=
\sum_{m+n=0}\widehat F(m,n).
\]
Since $\widehat F(0,0)=\frac{M}{N}$,
\[
\sum_{\substack{m+n=0\\(m,n)\neq(0,0)}}
|\widehat F(m,n)|
\geq
\frac{M}{N}.
\]
It follows that $\|\widehat F\|_1\geq\frac{2M}{N}$, and therefore
\[
\operatorname{FR}(F)
\geq
\frac{2\sqrt M}{N}.
\]
The product of the two lower bounds equals $2$, so their maximum is at least $\sqrt2$.
\end{proof}

For symmetric graph kernels, the full adjacency spectrum gives a stronger and more structural estimate. For an $N$ by $N$ matrix $M$, let
\[
\|M\|_*
=
\sum_{j=1}^N s_j(M)
\]
denote the nuclear norm, where $s_j(M)$ are the singular values. We use $\|M\|_{1,\operatorname{ent}}$ for the entrywise $\ell^1$ norm.

\begin{lemma}\label{lem:nuclear-entrywise}
For every complex matrix $M$,
\[
\|M\|_*
\leq
\|M\|_{1,\operatorname{ent}}.
\]
\end{lemma}

\begin{proof}
Write
\[
M
=
\sum_{j,k}M_{jk}E_{jk}.
\]
Each matrix unit $E_{jk}$ has nuclear norm one. The triangle inequality for the nuclear norm gives
\[
\|M\|_*
\leq
\sum_{j,k}|M_{jk}|\|E_{jk}\|_*
=
\sum_{j,k}|M_{jk}|.
\]
\end{proof}

Let $\lambda_1(A),\ldots,\lambda_N(A)$ be the eigenvalues of the adjacency matrix. The graph energy is
\begin{equation}\label{eq:graph-energy}
\mathcal E(G)
=
\sum_{j=1}^N|\lambda_j(A)|.
\end{equation}

\begin{theorem}[Energy lower bound]\label{thm:energy-lower-bound}
Let $G$ be a nonempty simple graph with $N$ vertices and $s$ edges. Then
\begin{equation}\label{eq:energy-fr-bound}
\operatorname{FR}_{\min}(G)
\geq
\frac{\mathcal E(G)}{\sqrt{2s}}.
\end{equation}
\end{theorem}

\begin{proof}
Fix a labeling $\sigma$ and let $A_\sigma$ be the corresponding adjacency matrix. Let
\[
U_{m,x}
=
\frac{1}{\sqrt N}e^{-\frac{2\pi imx}{N}}
\]
be the unitary Fourier matrix. Since $U$ is symmetric, the two-dimensional transform \eqref{eq:fourier-transform} can be written as
\[
\widehat{A_\sigma}
=
UA_\sigma U.
\]
Left and right multiplication by unitary matrices preserves singular values. Consequently,
\[
\|\widehat{A_\sigma}\|_*
=
\|A_\sigma\|_*.
\]
The matrix $A_\sigma$ is real symmetric, so its singular values are the absolute values of its eigenvalues. Relabeling does not change those eigenvalues, and therefore
\[
\|\widehat{A_\sigma}\|_*
=
\mathcal E(G).
\]
Lemma \ref{lem:nuclear-entrywise}, Parseval's identity, and \eqref{eq:adjacency-frobenius} now give
\[
\operatorname{FR}(A_\sigma)
=
\frac{\|\widehat{A_\sigma}\|_{1,\operatorname{ent}}}{\|\widehat{A_\sigma}\|_2}
\geq
\frac{\mathcal E(G)}{\sqrt{2s}}.
\]
Taking the minimum over $\sigma$ proves the result.
\end{proof}

There is also a lower bound arising only from the number of nonzero entries.

\begin{proposition}\label{prop:support-lower-bound}
Let $G$ be a nonempty simple graph with $N$ vertices and $s$ edges. For every labeling $\sigma$,
\begin{equation}\label{eq:support-lower}
\operatorname{FR}(A_\sigma)
\geq
\frac{N}{\sqrt{2s}}.
\end{equation}
Consequently,
\begin{equation}\label{eq:combined-lower}
\operatorname{FR}_{\min}(G)
\geq
\max\{
\frac{\mathcal E(G)}{\sqrt{2s}},
\frac{N}{\sqrt{2s}}
\}.
\end{equation}
\end{proposition}

\begin{proof}
This is the first estimate in Proposition \ref{prop:loop-free-bound} with $M=2s$.
\end{proof}

\begin{corollary}[Universal graph bound]\label{cor:sqrt-two}
Every nonempty simple graph satisfies
\begin{equation}\label{eq:sqrt-two}
\operatorname{FR}_{\min}(G)
\geq
\sqrt2.
\end{equation}
The bound is sharp for every even $N$.
\end{corollary}

\begin{proof}
Let $p_1,\ldots,p_a>0$ be the positive adjacency eigenvalues and let $-q_1,\ldots,-q_b<0$ be the negative adjacency eigenvalues. Since the adjacency matrix has trace zero,
\[
\sum_i p_i
=
\sum_j q_j
=
\frac{\mathcal E(G)}{2}.
\]
Also,
\[
\sum_i p_i^2+
\sum_jq_j^2
=
2s.
\]
Therefore
\[
2s
\leq
\left(\sum_i p_i\right)^2+
\left(\sum_jq_j\right)^2
=
\frac{\mathcal E(G)^2}{2}.
\]
Thus $\mathcal E(G)\geq2\sqrt s$, and Theorem \ref{thm:energy-lower-bound} gives \eqref{eq:sqrt-two}. Sharpness follows from the parity-labeled complete bipartite graph in Corollary \ref{cor:complete-turan}.
\end{proof}

\begin{proposition}[Equality in the energy estimate]\label{prop:energy-equality}
Let $G$ be a nonempty simple graph with $s$ edges. Then
\[
\mathcal E(G)
\geq
2\sqrt s.
\]
Equality holds if and only if $G$ consists of one complete bipartite component together with any number of isolated vertices.
\end{proposition}

\begin{proof}
The inequality was proved in Corollary \ref{cor:sqrt-two}. Equality forces equality in both estimates
\[
\sum_i p_i^2
\leq
\left(\sum_i p_i\right)^2
\]
and
\[
\sum_jq_j^2
\leq
\left(\sum_jq_j\right)^2.
\]
Thus, the nonzero adjacency spectrum consists of one positive and one negative eigenvalue, so the adjacency matrix has rank two. There can be only one component containing an edge.

Let $H$ be that component. A connected simple graph of adjacency rank two is complete bipartite. One direct argument is as follows. The graph contains no triangle and no induced path on four vertices, because the corresponding adjacency matrices have ranks three and four. Fix an edge $uv$. Every vertex is adjacent to $u$ or to $v$, since otherwise a shortest path to the set $\{u,v\}$ produces an induced path on four vertices. No vertex is adjacent to both $u$ and $v$, since the graph is triangle-free. If a vertex adjacent to $u$ and a vertex adjacent to $v$ were not adjacent, those four vertices would form an induced path on four vertices. Hence $H$ is complete bipartite.

Conversely, $K_{a,b}$ has nonzero adjacency eigenvalues $\sqrt{ab}$ and $-\sqrt{ab}$. Since it has $ab$ edges, its energy equals $2\sqrt s$. Isolated vertices do not change either quantity.
\end{proof}

\section{Circulant graphs and exact computations}\label{sec:circulant}

Let $S\subset\mathbb Z_N$ be nonempty and satisfy $0\notin S$ and $S=-S$. The circulant graph $\operatorname{Cay}(\mathbb Z_N,S)$ has adjacency relation
\[
x\sim y
\text{if and only if}
y-x\in S.
\]
Its natural labeling is the identity labeling of $\mathbb Z_N$.

\begin{theorem}\label{thm:circulant-optimal}
Let $G=\operatorname{Cay}(\mathbb Z_N,S)$ be a circulant graph. Under the natural labeling,
\begin{equation}\label{eq:circulant-fr-energy}
\operatorname{FR}(A)
=
\frac{\mathcal E(G)}{\sqrt{2s}}.
\end{equation}
Consequently, the natural labeling is globally optimal and
\begin{equation}\label{eq:circulant-min}
\operatorname{FR}_{\min}(G)
=
\frac{\mathcal E(G)}{\sqrt{2s}}.
\end{equation}
\end{theorem}

\begin{proof}
Write $A(x,y)=1_S(y-x)$. With $t=y-x$,
\begin{align*}
\widehat A(m,n)
&=
\frac{1}{N}
\sum_{x,t\in\mathbb Z_N}
1_S(t)e^{-\frac{2\pi i(mx+n(x+t))}{N}}\\
&=
\delta_{m+n,0}
\sum_{t\in S}e^{-\frac{2\pi int}{N}}.
\end{align*}
The quantities
\[
\lambda_n
=
\sum_{t\in S}e^{-\frac{2\pi int}{N}}
\]
are the adjacency eigenvalues, up to an immaterial reindexing. Thus, $\widehat A$ is supported on the anti-diagonal $m+n=0$, and its entries there are precisely the eigenvalues. Hence
\[
\|\widehat A\|_1
=
\sum_{n\in\mathbb Z_N}|\lambda_n|
=
\mathcal E(G).
\]
Together with $\|A\|_2=\sqrt{2s}$, this proves \eqref{eq:circulant-fr-energy}. The lower bound in Theorem \ref{thm:energy-lower-bound} proves global optimality.
\end{proof}

\begin{remark}[Fourier algebra and group lifts]\label{rem:group-lift}
Let $a:\mathbb Z_N\to\mathbb C$, and use the Haar-normalized one-dimensional Fourier transform
\[
\widetilde a(\xi)
=
\frac{1}{N}
\sum_{x\in\mathbb Z_N}
a(x)e^{-\frac{2\pi i\xi x}{N}}.
\]
Its Fourier algebra norm is
\[
\|a\|_{A(\mathbb Z_N)}
=
\sum_{\xi\in\mathbb Z_N}|\widetilde a(\xi)|.
\]
If $A(x,y)=a(x-y)$ is the corresponding group-lift matrix, then the general group-lift identity gives
\begin{equation}\label{eq:group-lift-trace}
\|A\|_*
=
N\|a\|_{A(\mathbb Z_N)}
\end{equation}
\cite[Proposition 3.11]{HambardzumyanHatamiHatami}. For a circulant graph, this identity reads
\[
\|a\|_{A(\mathbb Z_N)}
=
\frac{\mathcal E(G)}{N}.
\]
Theorem \ref{thm:circulant-optimal} uses this trace-norm structure together with the two-dimensional Fourier representation to show that the natural cyclic labeling attains the graph-energy lower bound among all cyclic labelings of the abstract graph.
\end{remark}

The same calculation has a group-relative form. If $H$ is any finite abelian group and $G=\operatorname{Cay}(H,S)$, then the Fourier transform on $H\times H$ gives equality between the Fourier ratio and the graph-energy lower bound under the natural group labeling. The invariant in this paper is defined using cyclic coordinates, so Theorem \ref{thm:circulant-optimal} is stated for $H=\mathbb Z_N$. For a noncyclic group, one may instead define an invariant relative to the chosen group structure.

\subsection{Complete and balanced Tur\'an graphs}

\begin{corollary}\label{cor:complete-turan}
For $N\geq2$,
\begin{equation}\label{eq:complete-fr}
\operatorname{FR}_{\min}(K_N)
=
2\sqrt{1-\frac{1}{N}}.
\end{equation}
If $2\leq k\leq N$, $k$ divides $N$, and $T(N,k)$ is the balanced complete $k$-partite graph, then
\begin{equation}\label{eq:turan-fr}
\operatorname{FR}_{\min}(T(N,k))
=
2\sqrt{1-\frac{1}{k}}.
\end{equation}
In particular, for even $N$,
\begin{equation}\label{eq:bipartite-parity}
\operatorname{FR}_{\min}
(K_{\frac{N}{2},\frac{N}{2}})
=
\sqrt2.
\end{equation}
\end{corollary}

\begin{proof}
The complete graph is circulant with generating set $\mathbb Z_N\setminus\{0\}$. Its adjacency spectrum is $N-1$ with multiplicity one and $-1$ with multiplicity $N-1$. Thus,
\[
\mathcal E(K_N)=2(N-1)
\]
and
\[
2s=N(N-1).
\]
Formula \eqref{eq:complete-fr} follows from Theorem \ref{thm:circulant-optimal}.

For the Tur\'an graph, label the $k$ parts by congruence classes modulo $k$. Then $T(N,k)$ is a circulant graph. Its adjacency eigenvalues are
\[
N\left(1-\frac{1}{k}\right)
\]
with multiplicity one,
\[
-\frac{N}{k}
\]
with multiplicity $k-1$, and $0$ with multiplicity $N-k$. Thus,
\[
\mathcal E(T(N,k))
=
2N\left(1-\frac{1}{k}\right),
\]
while
\[
2s
=
N^2\left(1-\frac{1}{k}\right).
\]
Theorem \ref{thm:circulant-optimal} gives \eqref{eq:turan-fr}. Taking $k=2$ gives \eqref{eq:bipartite-parity}.
\end{proof}

The parity description is particularly simple. If the two parts are the even and odd residue classes, then
\[
A(x,y)
=
\frac{1-(-1)^{x-y}}{2}.
\]
The Fourier transform has only two nonzero coefficients, at $(0,0)$ and $(\frac{N}{2},\frac{N}{2})$. This shows directly that the Fourier ratio equals $\sqrt2$.

The complete graph and the balanced complete bipartite graph illustrate why edge density is not the governing quantity. The complete graph has the largest possible number of edges but Fourier ratio approaching $2$. The parity-labeled complete bipartite graph has half of the ordered pairs as edges and attains the universal lower bound $\sqrt2$. In both cases, the edge relation has a short additive description.

\subsection{Cycles}

The cycle $C_N$ is the circulant graph generated by $\{1,-1\}$. Its adjacency eigenvalues are
\[
2\cos\left(\frac{2\pi j}{N}\right),
0\leq j<N.
\]

\begin{corollary}\label{cor:cycle-exact}
For $N\geq3$,
\begin{equation}\label{eq:cycle-fr-exact}
\operatorname{FR}_{\min}(C_N)
=
\frac{\mathcal E(C_N)}{\sqrt{2N}},
\end{equation}
where
\begin{equation}\label{eq:cycle-energy}
\mathcal E(C_N)
=
\begin{cases}
2\csc(\frac{\pi}{2N}),&N\text{ is odd},\\
4\cot(\frac{\pi}{N}),&N\equiv0\pmod4,\\
4\csc(\frac{\pi}{N}),&N\equiv2\pmod4.
\end{cases}
\end{equation}
In particular,
\begin{equation}\label{eq:cycle-asymptotic}
\operatorname{FR}_{\min}(C_N)
=
\frac{2\sqrt2}{\pi}\sqrt N+o(\sqrt N).
\end{equation}
\end{corollary}

\begin{proof}
The first formula follows from Theorem \ref{thm:circulant-optimal}, since $C_N$ has $N$ edges. To evaluate the energy, use
\[
\sum_{j=a}^b\cos(j\theta)
=
\frac{
\sin(\frac{(b-a+1)\theta}{2})
\cos(\frac{(a+b)\theta}{2})
}{
\sin(\frac{\theta}{2})
}.
\]
The signs of $\cos(\frac{2\pi j}{N})$ change at the two quarter points of the cyclic interval. Splitting the sum at those points and applying the displayed identity gives the three cases in \eqref{eq:cycle-energy}. In every case,
\[
\mathcal E(C_N)
=
\frac{4N}{\pi}+o(N),
\]
and \eqref{eq:cycle-asymptotic} follows.
\end{proof}

The recovery theorem therefore gives a sufficient sample count of order $N$, up to the accuracy and logarithmic factors, for the naturally labeled cycle. This is one of the clearest examples in which a sparse graph has an efficiently recoverable additive representation but does not have bounded Fourier ratio.

\subsection{A labeling-sensitive complete bipartite example}

The parity labeling of $K_{\frac{N}{2},\frac{N}{2}}$ realizes the minimum $\sqrt2$. We now compute a very different labeling of the same abstract graph. This calculation also shows why the Fourier transform is not supported on a union of a few lines.

\begin{proposition}[Interval labeling of a balanced complete bipartite graph]\label{prop:interval-bipartite}
Let $N$ be even and set
\[
I
=
\{0,1,\ldots,\frac{N}{2}-1\}
\]
and
\[
I^c
=
\mathbb Z_N\setminus I.
\]
Let
\[
B_N(x,y)
=
1_I(x)1_{I^c}(y)+1_{I^c}(x)1_I(y)
\]
be the adjacency kernel of $K_{\frac{N}{2},\frac{N}{2}}$ under the interval labeling. Define
\[
a_m
=
\sum_{x=0}^{\frac{N}{2}-1}
e^{-\frac{2\pi imx}{N}}
\]
and
\[
S_N
=
\sum_{m\neq0}|a_m|.
\]
Then
\begin{equation}\label{eq:interval-exact}
\operatorname{FR}(B_N)
=
\frac{1}{\sqrt2}
+
\frac{2\sqrt2}{N^2}S_N^2.
\end{equation}
Moreover,
\begin{equation}\label{eq:interval-asymptotic}
S_N
=
\frac{N}{\pi}\log N+O(N),
\end{equation}
and therefore
\begin{equation}\label{eq:interval-fr-asymptotic}
\operatorname{FR}(B_N)
=
\frac{2\sqrt2}{\pi^2}(\log N)^2+O(\log N).
\end{equation}
\end{proposition}

\begin{proof}
Let
\[
b_m
=
\sum_{x\in I^c}e^{-\frac{2\pi imx}{N}}
=
N\delta_{m,0}-a_m.
\]
Factorization gives
\[
\widehat B_N(m,n)
=
\frac{1}{N}(a_mb_n+b_ma_n).
\]
At $(m,n)=(0,0)$ this equals $\frac{N}{2}$. If exactly one of $m,n$ is zero, it vanishes. If $m,n\neq0$, then $b_m=-a_m$ and $b_n=-a_n$, so
\[
\widehat B_N(m,n)
=
-\frac{2}{N}a_ma_n.
\]
Consequently,
\[
\|\widehat B_N\|_1
=
\frac{N}{2}
+
\frac{2}{N}
(\sum_{m\neq0}|a_m|)^2
=
\frac{N}{2}
+
\frac{2}{N}S_N^2.
\]
The graph has $\frac{N^2}{2}$ ordered edges, so
\[
\|\widehat B_N\|_2
=
\|B_N\|_2
=
\frac{N}{\sqrt2}.
\]
This proves \eqref{eq:interval-exact}.

The geometric-series identity gives, for $m\neq0$,
\[
|a_m|
=
\begin{cases}
0,&m\text{ is even},\\
\csc(\frac{\pi m}{N}),&m\text{ is odd}.
\end{cases}
\]
For $1\leq m\leq\frac{N}{2}$,
\[
\csc\left(\frac{\pi m}{N}\right)
=
\frac{N}{\pi m}+O\left(\frac{m}{N}\right).
\]
Symmetry about $\frac{N}{2}$ gives
\[
S_N
=
2
\sum_{\substack{1\leq m<\frac{N}{2}\\m\text{ is odd}}}
\csc\left(\frac{\pi m}{N}\right)
+O(1).
\]
Using the preceding expansion and the harmonic sum over odd integers gives
\[
S_N
=
\frac{N}{\pi}\log N+O(N).
\]
Substitution into \eqref{eq:interval-exact} proves \eqref{eq:interval-fr-asymptotic}.
\end{proof}

\begin{corollary}\label{cor:labeling-instability-bipartite}
For the balanced complete bipartite graph,
\[
\liminf_{\substack{N\longrightarrow\infty\\2\mid N}}
\frac{
\operatorname{FR}_{\max}
(K_{\frac{N}{2},\frac{N}{2}})
}{
(\log N)^2
\operatorname{FR}_{\min}
(K_{\frac{N}{2},\frac{N}{2}})
}
\geq
\frac{2}{\pi^2}.
\]
In particular, the labeling-instability quotient is unbounded.
\end{corollary}

\begin{proof}
The interval labeling gives a lower bound for $\operatorname{FR}_{\max}$, while the parity labeling realizes $\operatorname{FR}_{\min}=\sqrt2$.
\end{proof}

This example has a useful recovery interpretation. The parity labeling has bounded Fourier ratio. The interval labeling has Fourier ratio of order $(\log N)^2$, so the structural factor $r^2$ in the sufficient sample bound is of order $(\log N)^4$. Both labelings remain substantially more compressible than the random-phase scale of order $N$, but they are not equivalent from the sampling point of view.

\section{Further structured kernels and extensions}\label{sec:extensions}

The preceding examples are based on translation invariance and congruence classes. Similar estimates hold whenever a kernel is generated by a small number of affine relations.

\begin{proposition}\label{prop:affine-kernels}
Let $a_j\in\mathbb Z_N^\times$, $b_j\in\mathbb Z_N$, and $c_j\in\mathbb C$ for $1\leq j\leq K$. Suppose
\[
F(x,y)
=
\sum_{j=1}^Kc_j1_{\{y=a_jx+b_j\}}
\]
is nonzero. Then
\begin{equation}\label{eq:affine-upper}
\operatorname{FR}(F)
\leq
\sqrt{KN}.
\end{equation}
\end{proposition}

\begin{proof}
For one affine relation,
\[
\widehat{1_{\{y=ax+b\}}}(m,n)
=
e^{-\frac{2\pi inb}{N}}\delta_{m+an,0}.
\]
Hence, $\widehat F$ is supported on the union of at most $K$ lines, a set of cardinality at most $KN$. The support estimate \eqref{eq:support-upper-bound} proves the result.
\end{proof}

If $F$ is intended to be the adjacency indicator of a simple graph, the affine relations must satisfy additional conditions. Their supports should not overlap in a way that creates weights larger than one, the union should be symmetric under interchange of $x$ and $y$, and loops should be excluded. Without these assumptions, Proposition \ref{prop:affine-kernels} is a statement about a weighted or directed kernel rather than a simple graph.

\subsection{Fourier algebra structure of low-ratio graph kernels}\label{sec:fourier-algebra}

The preceding affine examples give explicit constructions. A converse structural statement follows from the quantitative Cohen idempotent theorem. We first record the normalization. For $F:\mathbb Z_N^2\to\mathbb C$, define the Haar-normalized Fourier transform by
\[
\widetilde F(m,n)
=
\frac{1}{N^2}
\sum_{x,y\in\mathbb Z_N}
F(x,y)e^{-\frac{2\pi i(mx+ny)}{N}}.
\]
Thus,
\[
\widetilde F(m,n)
=
\frac{1}{N}\widehat F(m,n).
\]
The Fourier algebra norm is
\[
\|F\|_{A(\mathbb Z_N^2)}
=
\sum_{m,n\in\mathbb Z_N}|\widetilde F(m,n)|
=
\frac{1}{N}\|\widehat F\|_1.
\]

Green and Sanders proved a quantitative finite-group form of the Cohen idempotent theorem, and Sanders later obtained the following sharper dependence \cite{GreenSanders,SandersCohen}.

\begin{theorem}\label{thm:quantitative-cohen}
There is a function $\Phi:[1,\infty)\to[1,\infty)$ with
\[
\Phi(M)
=
\exp(M^{4+o(1)})
\]
as $M\longrightarrow\infty$ with the following property. Let $H$ be a finite abelian group, and let $F:H\to\mathbb Z$ satisfy
\[
\|F\|_{A(H)}
\leq
M.
\]
Then there are cosets $W_1,\ldots,W_L$ of subgroups of $H$ and signs $\epsilon_j\in\{-1,1\}$ such that
\begin{equation}\label{eq:cohen-decomposition}
F
=
\sum_{j=1}^L\epsilon_j1_{W_j}
\end{equation}
and
\[
L
\leq
\Phi(M).
\]
\end{theorem}

\begin{corollary}[Coset structure of low-ratio Boolean kernels]\label{cor:coset-structure}
Let $F:\mathbb Z_N^2\to\{0,1\}$ be nonzero, and let
\[
\alpha
=
\frac{1}{N^2}
\sum_{x,y\in\mathbb Z_N}F(x,y).
\]
If
\[
\operatorname{FR}(F)
\leq
r,
\]
then $F$ has a representation of the form \eqref{eq:cohen-decomposition}, where the $W_j$ are cosets of subgroups of $\mathbb Z_N^2$ and
\begin{equation}\label{eq:coset-length}
L
\leq
\Phi(\sqrt\alpha r).
\end{equation}
Consequently, if $\operatorname{FR}_{\min}(G)\leq r$, then some labeling of the adjacency kernel of $G$ has such a decomposition with
\[
\alpha
=
\frac{2s}{N^2}.
\]
\end{corollary}

\begin{proof}
Since $F$ is Boolean,
\[
\|F\|_2
=
N\sqrt\alpha.
\]
Since $F$ is a nonzero integer-valued function, Fourier inversion also gives
\[
1
\leq
\|F\|_\infty
\leq
\|F\|_{A(\mathbb Z_N^2)}.
\]
Therefore,
\begin{align*}
\|F\|_{A(\mathbb Z_N^2)}
&=
\frac{1}{N}\|\widehat F\|_1\\
&=
\frac{1}{N}
\operatorname{FR}(F)\|F\|_2\\
&=
\sqrt\alpha\operatorname{FR}(F)\\
&\leq
\sqrt\alpha r.
\end{align*}
Theorem \ref{thm:quantitative-cohen} gives the decomposition and \eqref{eq:coset-length}. The graph statement follows by choosing a labeling that realizes the minimum Fourier ratio.
\end{proof}

Corollary \ref{cor:coset-structure} gives a dimension-free inverse theorem. A Boolean kernel with bounded Fourier ratio is an exact signed combination of a bounded number of additive pieces, independently of $N$. If $N$ is prime, then $\mathbb Z_N^2$ is a two-dimensional vector space over the field with $N$ elements. Its subgroup cosets are points, affine lines, and the whole plane. For composite $N$, the cosets include more general congruence structures.

The conclusion is exact, but it is not yet a canonical graph decomposition. The signs allow cancellation, and an individual coset need not be symmetric under interchange of the two coordinates or avoid the diagonal. The Boolean, symmetric, and loop-free properties emerge only after the full signed sum is formed. A graph-specific inverse theory should determine how much additional structure these constraints force.

The nuclear-norm argument also extends beyond simple graphs.

\begin{proposition}[A matrix lower bound]\label{prop:matrix-lower-bound}
Let $F$ be a nonzero complex $N$ by $N$ matrix, regarded as a function on $\mathbb Z_N^2$. Then
\begin{equation}\label{eq:matrix-nuclear-bound}
\operatorname{FR}(F)
\geq
\frac{\|F\|_*}{\|F\|_2}.
\end{equation}
\end{proposition}

\begin{proof}
The two-dimensional Fourier transform has the matrix form $\widehat F=UFU$, where $U$ is unitary. Thus, $\|\widehat F\|_*=\|F\|_*$. Lemma \ref{lem:nuclear-entrywise} and Parseval's identity give
\[
\operatorname{FR}(F)
=
\frac{\|\widehat F\|_{1,\operatorname{ent}}}{\|\widehat F\|_2}
\geq
\frac{\|F\|_*}{\|F\|_2}.
\]
\end{proof}

For a Hermitian weighted graph, the numerator in \eqref{eq:matrix-nuclear-bound} is the sum of the absolute values of the weighted adjacency eigenvalues. For a directed graph, it is the sum of the singular values. Signed graphs are included without modification. The recovery theorem is homogeneous, so multiplying all weights by a common scalar does not change the Fourier ratio.

There is also a simple fixed-labeling perturbation estimate.

\begin{proposition}[Stability under matrix perturbation]\label{prop:fr-perturbation}
Let $F$ and $G$ be nonzero kernels on $\mathbb Z_N^2$. Suppose
\[
\delta
=
\frac{\|G-F\|_2}{\|F\|_2}
<1.
\]
Then
\begin{equation}\label{eq:fr-perturbation}
\operatorname{FR}(G)
\leq
\frac{\operatorname{FR}(F)+N\delta}{1-\delta}.
\end{equation}
\end{proposition}

\begin{proof}
By Cauchy--Schwarz and Parseval's identity,
\[
\|\widehat{G-F}\|_1
\leq
N\|G-F\|_2.
\]
Therefore,
\[
\|\widehat G\|_1
\leq
\operatorname{FR}(F)\|F\|_2
+
N\|G-F\|_2.
\]
Also,
\[
\|G\|_2
\geq
\|F\|_2-\|G-F\|_2.
\]
Dividing the two estimates proves \eqref{eq:fr-perturbation}.
\end{proof}

If two simple graphs in the same labeling differ by $t$ undirected edges and the first graph has $s$ edges, then
\[
\delta
=
\sqrt{\frac{t}{s}}.
\]
The resulting bound is crude because of the factor $N$, but it gives a rigorous form of edit stability when the perturbation is sufficiently small. The more delicate question is whether two nearby abstract graphs admit one labeling that is simultaneously good for both.

The Fourier ratio also has an exact entropy interpretation. Let
\[
p_{m,n}
=
\frac{|\widehat F(m,n)|^2}{\|\widehat F\|_2^2}.
\]
Then $p$ is a probability distribution on $\mathbb Z_N^2$. Its R\'enyi entropy of order $\frac{1}{2}$ is
\[
H_{\frac{1}{2}}(p)
=
2\log(\sum_{m,n}\sqrt{p_{m,n}})
=
2\log\operatorname{FR}(F).
\]
Thus, $\operatorname{FR}(F)^2$ may be viewed as an entropic effective dimension. This is consistent with the role of the square of the Fourier ratio in recovery and metric entropy estimates \cite{IosevichHovhannisyanKeyshamsVagharshakyan}.

\section{Spectral-projector kernels and harmonic complexity}\label{sec:projectors}

Let $L=D-A$ be the combinatorial Laplacian of $G$. For an eigenvalue $\lambda$, let $\Pi_\lambda$ be the orthogonal projector onto the corresponding eigenspace, and let
\[
m(\lambda)
=
\operatorname{rank}(\Pi_\lambda)
\]
be its multiplicity. A labeling $\sigma:\mathbb Z_N\to V(G)$ identifies $\Pi_\lambda$ with a kernel $\Pi_{\lambda,\sigma}$ on $\mathbb Z_N^2$. Define
\begin{equation}\label{eq:projector-fr-min}
\operatorname{FR}_{\min}(\Pi_\lambda)
=
\min_\sigma\operatorname{FR}(\Pi_{\lambda,\sigma}).
\end{equation}

The projector kernel is intrinsic before the labeling is chosen. It is also independent of the choice of an orthonormal basis inside the eigenspace. The Fourier ratio, however, depends on the coordinates placed on the vertices, just as it does for the adjacency kernel.

\begin{remark}\label{rem:projector-oracle}
All recovery statements concerning projectors will use samples of the projector kernel $\Pi_{\lambda,\sigma}(x,y)$ itself. Samples of the adjacency matrix do not directly provide these values. Deriving projector estimates from incomplete adjacency data is a distinct perturbation and inverse problem that is not addressed by the Fourier-ratio theorem alone.
\end{remark}

\subsection{The multiplicity lower bound}

\begin{theorem}\label{thm:projector-lower-bound}
For every graph, every Laplacian eigenvalue $\lambda$, and every vertex labeling $\sigma$,
\begin{equation}\label{eq:projector-lower}
\operatorname{FR}(\Pi_{\lambda,\sigma})
\geq
\sqrt{m(\lambda)}.
\end{equation}
Consequently,
\begin{equation}\label{eq:projector-min-lower}
\operatorname{FR}_{\min}(\Pi_\lambda)
\geq
\sqrt{m(\lambda)}.
\end{equation}
\end{theorem}

\begin{proof}
The projector has $m(\lambda)$ singular values equal to one and all remaining singular values equal to zero. Thus,
\[
\|\Pi_{\lambda,\sigma}\|_*
=
m(\lambda)
\]
and
\[
\|\Pi_{\lambda,\sigma}\|_2
=
\sqrt{m(\lambda)}.
\]
As in the proof of Theorem \ref{thm:energy-lower-bound}, the two-dimensional Fourier transform preserves singular values. Lemma \ref{lem:nuclear-entrywise} therefore gives
\[
\|\widehat{\Pi_{\lambda,\sigma}}\|_1
\geq
m(\lambda).
\]
Division by the Frobenius norm proves \eqref{eq:projector-lower}. Minimizing over $\sigma$ gives \eqref{eq:projector-min-lower}.
\end{proof}

If the graph has $c$ connected components, the zero eigenspace has multiplicity $c$. Theorem \ref{thm:projector-lower-bound} then gives
\[
\operatorname{FR}_{\min}(\Pi_0)
\geq
\sqrt c.
\]
For a connected graph, the zero projector is the constant kernel $\frac{1}{N}J$ and its Fourier ratio equals one.

\subsection{Circulant graphs}

\begin{theorem}[Exact projector complexity for circulant graphs]\label{thm:circulant-projectors}
Let $G$ be a circulant graph on $\mathbb Z_N$ with its natural labeling. For a Laplacian eigenvalue $\lambda$, let
\[
S_\lambda
=
\{\xi\in\mathbb Z_N:\lambda_\xi=\lambda\}
\]
be its frequency set. Then
\begin{equation}\label{eq:circulant-projector-exact}
\operatorname{FR}(\Pi_\lambda)
=
\sqrt{|S_\lambda|}
=
\sqrt{m(\lambda)}.
\end{equation}
Hence, the natural labeling simultaneously minimizes $\operatorname{FR}(A)$ and $\operatorname{FR}(\Pi_\lambda)$ for every $\lambda$.
\end{theorem}

\begin{proof}
The normalized characters
\[
\frac{1}{\sqrt N}e^{\frac{2\pi i\xi x}{N}}
\]
diagonalize the Laplacian. Therefore,
\[
\Pi_\lambda(x,y)
=
\frac{1}{N}
\sum_{\xi\in S_\lambda}
e^{\frac{2\pi i\xi(x-y)}{N}}.
\]
A direct computation gives
\[
\widehat{\Pi_\lambda}(m,n)
=
1_{S_\lambda}(m)\delta_{m+n,0}.
\]
Thus, $\widehat{\Pi_\lambda}$ has exactly $|S_\lambda|$ nonzero coefficients, each of modulus one. Its $\ell^1$ and $\ell^2$ norms are $|S_\lambda|$ and $\sqrt{|S_\lambda|}$, respectively. This proves \eqref{eq:circulant-projector-exact}. Theorem \ref{thm:projector-lower-bound} proves optimality for each projector, while Theorem \ref{thm:circulant-optimal} proves optimality for the adjacency matrix.
\end{proof}

As with the adjacency calculation, the same proof applies to a Cayley graph on an arbitrary finite abelian group when the Fourier transform on that group is used.

\begin{corollary}\label{cor:projector-examples}
Under the natural cyclic labeling, every Laplacian projector of $C_N$ has Fourier ratio $1$ or $\sqrt2$.

For $K_N$, the projectors corresponding to the eigenvalues $0$ and $N$ have Fourier ratios $1$ and $\sqrt{N-1}$, respectively.

For $K_{\frac{N}{2},\frac{N}{2}}$ with the parity labeling, the projectors corresponding to the eigenvalues $0$, $N$, and $\frac{N}{2}$ have Fourier ratios $1$, $1$, and $\sqrt{N-2}$, respectively.
\end{corollary}

\begin{proof}
The nontrivial Laplacian eigenvalues of a cycle have multiplicity two, except for the eigenvalues arising from the frequencies $0$ and, when $N$ is even, $\frac{N}{2}$. The Laplacian spectrum of $K_N$ is $0$ with multiplicity one and $N$ with multiplicity $N-1$. The Laplacian spectrum of $K_{\frac{N}{2},\frac{N}{2}}$ is $0$ with multiplicity one, $N$ with multiplicity one, and $\frac{N}{2}$ with multiplicity $N-2$. The result follows from Theorem \ref{thm:circulant-projectors}.
\end{proof}

\subsection{Edge complexity versus harmonic complexity}

The Fourier ratio framework gives two distinct notions of graph complexity. The adjacency Fourier ratio measures the additive spectral compressibility of the edge relation itself. The projector Fourier ratio measures the additive spectral compressibility of one harmonic sector of the Laplacian.

The complete graph is simple at the level of edges and complicated at the level of its nontrivial harmonic sector. We have
\[
\operatorname{FR}_{\min}(K_N)
=
2\sqrt{1-\frac{1}{N}},
\]
while
\[
\operatorname{FR}_{\min}(\Pi_N)
=
\sqrt{N-1}.
\]
The cycle exhibits the opposite behavior. Its adjacency Fourier ratio is asymptotic to $\frac{2\sqrt2}{\pi}\sqrt N$, while every projector Fourier ratio is at most $\sqrt2$.

The balanced complete bipartite graph provides a third pattern. Its parity-labeled adjacency kernel attains the universal minimum $\sqrt2$, but the projector onto the eigenspace of eigenvalue $\frac{N}{2}$ has Fourier ratio $\sqrt{N-2}$. Thus, even the strongest possible additive compression of the edge relation does not force every harmonic sector to be compressible.

These examples show that adjacency complexity and harmonic complexity should not be identified. They answer different questions. A small adjacency Fourier ratio gives access to a stable approximation of the displayed edge kernel. A small projector Fourier ratio gives access, in a different observation model, to one spectral band. The two notions become simultaneously optimal for circulant graphs because the same character basis diagonalizes the adjacency matrix and all Laplacian projectors.

The multiplicity lower bound suggests that $m(\lambda)$ is an effective harmonic dimension. This interpretation parallels the role of $\operatorname{FR}(F)^2$ as an effective dimension in recovery and metric entropy. A high-multiplicity eigenspace cannot have a projector Fourier ratio smaller than the square root of its dimension, regardless of how the vertices are relabeled.

\section{Strongly regular graphs and simultaneous minimization}\label{sec:srg}

For a strongly regular graph, the nontrivial spectral projectors are affine combinations of the adjacency matrix, the identity matrix, and the all-ones matrix. This creates a direct link between adjacency Fourier complexity and harmonic Fourier complexity.

Let $G$ be a connected strongly regular graph with parameters $(v,d,\lambda,\mu)$. Let $d>r>s$ be the three adjacency eigenvalues, with multiplicities $1,f,g$. The two nonzero Laplacian eigenvalues are $d-r$ and $d-s$. The nontrivial adjacency eigenvalues are the roots of
\[
t^2-(\lambda-\mu)t-(d-\mu)=0.
\]
Thus,
\[
r
=
\frac{\lambda-\mu+\sqrt{(\lambda-\mu)^2+4(d-\mu)}}{2}
\]
and
\[
s
=
\frac{\lambda-\mu-\sqrt{(\lambda-\mu)^2+4(d-\mu)}}{2}.
\]
The multiplicities are determined by $v,d,r,s$ and may be found from the trace and dimension identities; see \cite{BrouwerHaemers}.

Since the graph is regular, the projector onto the constant eigenspace is $\frac{1}{v}J$. The two nontrivial projectors are
\begin{equation}\label{eq:srg-projector-r}
\Pi_{d-r}
=
\frac{1}{r-s}
(
A-sI-\frac{d-s}{v}J
)
\end{equation}
and
\begin{equation}\label{eq:srg-projector-s}
\Pi_{d-s}
=
\frac{1}{s-r}
(
A-rI-\frac{d-r}{v}J
).
\end{equation}

\begin{proposition}\label{prop:srg-link}
Fix a labeling of a connected strongly regular graph with parameters and notation as above. Then
\begin{equation}\label{eq:srg-bound-r}
\operatorname{FR}(\Pi_{d-r})
\leq
\frac{
\sqrt{vd}\operatorname{FR}(A)-d+|s|(v-1)
}{
(r-s)\sqrt f
}
\end{equation}
and
\begin{equation}\label{eq:srg-bound-s}
\operatorname{FR}(\Pi_{d-s})
\leq
\frac{
\sqrt{vd}\operatorname{FR}(A)-d+|r|(v-1)
}{
(r-s)\sqrt g
}.
\end{equation}
\end{proposition}

\begin{proof}
For a matrix indexed by $\mathbb Z_v$,
\[
\widehat I(m,n)
=
\delta_{m+n,0}
\]
and
\[
\widehat J(m,n)
=
v\delta_{m,0}\delta_{n,0}.
\]
Since $G$ is $d$-regular,
\[
\widehat A(0,0)=d.
\]
Taking the Fourier transform in \eqref{eq:srg-projector-r} gives
\[
\widehat{\Pi_{d-r}}(0,0)=0.
\]
For $m+n=0$ and $(m,n)\neq(0,0)$,
\[
\widehat{\Pi_{d-r}}(m,n)
=
\frac{\widehat A(m,-m)-s}{r-s},
\]
while for $m+n\neq0$,
\[
\widehat{\Pi_{d-r}}(m,n)
=
\frac{\widehat A(m,n)}{r-s}.
\]
Therefore,
\[
\|\widehat{\Pi_{d-r}}\|_1
\leq
\frac{
\|\widehat A\|_1-d+|s|(v-1)
}{r-s}.
\]
Since $A$ has $vd$ nonzero entries,
\[
\|\widehat A\|_1
=
\sqrt{vd}\operatorname{FR}(A).
\]
Also, $\|\Pi_{d-r}\|_2=\sqrt f$. This proves \eqref{eq:srg-bound-r}. The second estimate follows from \eqref{eq:srg-projector-s} in the same way.
\end{proof}

These estimates are one-sided and use the same labeling for the adjacency matrix and the projector. They do not imply that a labeling minimizing $\operatorname{FR}(A)$ also minimizes either projector. They do show that, within the strongly regular family, a good adjacency labeling supplies an explicit upper bound for the harmonic complexity in that same coordinate system.

Taking a labeling that minimizes the adjacency Fourier ratio gives, for example,
\[
\operatorname{FR}_{\min}(\Pi_{d-r})
\leq
\frac{
\sqrt{vd}\operatorname{FR}_{\min}(G)-d+|s|(v-1)
}{
(r-s)\sqrt f
}.
\]
The analogous estimate holds for $\Pi_{d-s}$.

\subsection{The Petersen graph}

The Petersen graph gives a concrete example in which adjacency and harmonic minimizers are different. We use the Kneser description. The vertices are the two-element subsets of $\{0,1,2,3,4\}$, and two vertices are adjacent when the corresponding subsets are disjoint. We number the vertices in lexicographic order:
\[
\begin{aligned}
0&=\{0,1\},&1&=\{0,2\},&2&=\{0,3\},&3&=\{0,4\},&4&=\{1,2\},\\
5&=\{1,3\},&6&=\{1,4\},&7&=\{2,3\},&8&=\{2,4\},&9&=\{3,4\}.
\end{aligned}
\]
The adjacency eigenvalues are $3$, $1$, and $-2$, with multiplicities $1$, $5$, and $4$. The Laplacian eigenvalues are $0$, $2$, and $5$. The two nontrivial projectors are
\begin{equation}\label{eq:petersen-projectors}
\Pi_2
=
\frac{1}{3}
(A+2I-\frac{1}{2}J)
\end{equation}
and
\[
\Pi_5
=
\frac{1}{3}
(-A+I+\frac{1}{5}J).
\]

For a kernel $M$, let
\[
\mathcal S_M
=
\{\sigma\in S_{10}:\operatorname{FR}(M_\sigma)=\operatorname{FR}_{\min}(M)\}.
\]

\begin{proposition}[Exhaustive Petersen computation]\label{prop:petersen-computation}
For the Petersen graph with the convention above,
\[
\operatorname{FR}_{\min}(A)
=
3.6806729989\ldots,
\]
\[
\operatorname{FR}_{\min}(\Pi_2)
=
3.5694013108\ldots,
\]
and
\[
\operatorname{FR}_{\min}(\Pi_5)
=
3.4907119850\ldots.
\]
The numbers of minimizing labelings are
\[
|\mathcal S_A|=960,
|\mathcal S_{\Pi_2}|=240,
|\mathcal S_{\Pi_5}|=240.
\]
Moreover,
\[
\mathcal S_A\cap\mathcal S_{\Pi_2}
=
\mathcal S_A\cap\mathcal S_{\Pi_5}
=
\varnothing,
\]
while
\[
\mathcal S_{\Pi_2}
=
\mathcal S_{\Pi_5}.
\]
\end{proposition}

\begin{proof}
The automorphism group of the Petersen graph is the group induced by permutations of $\{0,1,2,3,4\}$ and has order $120$. Its action on the set of labelings is free. It is therefore enough to evaluate one labeling from each of the
\[
\frac{10!}{120}=30240
\]
orbits.

For every orbit representative $\sigma$ and each of the three matrices $A$, $\Pi_2$, and $\Pi_5$, the computation forms $UM_\sigma U$, sums the absolute values of its entries, and divides by the Frobenius norm. The minimizing orbit counts are $8$, $2$, and $2$, respectively. Multiplication by $120$ gives the displayed numbers of labelings.

The computation was carried out with complex long-double arithmetic, and values within $10^{-12}$ were grouped. The two projector minima occur on the same two orbits, while none of the eight adjacency-minimizing orbits is one of those two orbits. The next distinct Fourier-ratio values exceed the three minima by more than $0.143$, $0.295$, and $0.329$, respectively. Thus, the classification of the minimizing orbits is insensitive to numerical grouping at the precision used to report the displayed values. The vertex convention, the projector formulas, and the orbit count above determine the computation without any external graph data. Complete executable code is given in Appendix \ref{app:petersen-code}.
\end{proof}

One adjacency-minimizing labeling is
\[
(0,1,2,5,9,3,8,7,4,6),
\]
for which
\[
\operatorname{FR}(\Pi_2)=3.8644931738\ldots
\]
and
\[
\operatorname{FR}(\Pi_5)=3.8206347176\ldots.
\]
One labeling minimizing both projectors is
\[
(0,1,4,5,7,8,9,2,3,6),
\]
for which
\[
\operatorname{FR}(A)=3.8238833919\ldots.
\]

The Petersen computation shows that simultaneous optimality for circulant graphs is a genuine structural feature and not a general property of spectral projectors. It also suggests two different optimization problems. One may seek a labeling that minimizes the edge relation, or one may seek a labeling that makes a selected collection of harmonic sectors simultaneously compressible.

\section{Projector sampling and heat-kernel approximation}\label{sec:projector-recovery}

Theorem \ref{thm:fr-recovery} applies directly to a projector kernel in the observation model of Remark \ref{rem:projector-oracle}. Orthogonal projectors satisfy
\[
|\Pi_\lambda(x,y)|^2
\leq
\Pi_\lambda(x,x)\Pi_\lambda(y,y)
\leq
1,
\]
because they are positive semidefinite and have diagonal entries between zero and one.

\begin{corollary}[Stable recovery of a sampled projector kernel]\label{cor:projector-recovery}
Fix a labeling $\sigma$ and an eigenvalue $\lambda$. Suppose
\[
\operatorname{FR}(\Pi_{\lambda,\sigma})
\leq
r.
\]
If entries of $\Pi_{\lambda,\sigma}$ are observed on a Bernoulli-$p$ subset $X\subset\mathbb Z_N^2$ and \eqref{eq:graph-sample-bound} holds, then, with probability at least $1-e^{-cpN^2}$, a Fourier-$\ell^1$ minimizer $\Pi_\lambda^\sharp$ satisfies
\begin{equation}\label{eq:projector-recovery-error}
\|\Pi_\lambda^\sharp-\Pi_{\lambda,\sigma}\|_2
\leq
11.47\epsilon\sqrt{m(\lambda)}.
\end{equation}
\end{corollary}

We next isolate the deterministic step that converts projector approximations into a heat-kernel approximation. This statement makes clear the role of a common labeling and the effect of omitted eigenspaces.

\begin{proposition}[Heat-kernel approximation from projector approximations]\label{prop:heat-kernel}
Let
\[
L
=
\sum_\lambda\lambda\Pi_\lambda
\]
be the Laplacian spectral decomposition in one fixed labeling. Fix $t\geq0$. Let $\Lambda$ be a set of eigenvalues, and suppose matrices $\widetilde\Pi_\lambda$ satisfy
\[
\|\widetilde\Pi_\lambda-\Pi_\lambda\|_2
\leq
\delta_\lambda
\]
for every $\lambda\in\Lambda$. Define
\[
\widetilde H_{t,\Lambda}
=
\sum_{\lambda\in\Lambda}
e^{-t\lambda}\widetilde\Pi_\lambda.
\]
Then
\begin{equation}\label{eq:heat-kernel-bound}
\|\widetilde H_{t,\Lambda}-e^{-tL}\|_2
\leq
\sum_{\lambda\in\Lambda}
e^{-t\lambda}\delta_\lambda
+
(
\sum_{\lambda\notin\Lambda}
e^{-2t\lambda}m(\lambda)
)^{\frac{1}{2}}.
\end{equation}
\end{proposition}

\begin{proof}
Write
\[
\widetilde H_{t,\Lambda}-e^{-tL}
=
\sum_{\lambda\in\Lambda}
e^{-t\lambda}
(\widetilde\Pi_\lambda-\Pi_\lambda)
-
\sum_{\lambda\notin\Lambda}
e^{-t\lambda}\Pi_\lambda.
\]
The first sum is bounded by the triangle inequality. The projectors in the second sum are pairwise orthogonal in the Frobenius inner product, and $\|\Pi_\lambda\|_2^2=m(\lambda)$. Its squared Frobenius norm is therefore
\[
\sum_{\lambda\notin\Lambda}
e^{-2t\lambda}m(\lambda).
\]
Combining the two estimates proves \eqref{eq:heat-kernel-bound}.
\end{proof}

\begin{corollary}[Heat-kernel recovery in the projector-sampling model]\label{cor:heat-recovery}
Fix $t\geq0$, one labeling, and a set $\Lambda$ of Laplacian eigenvalues such that
\[
\operatorname{FR}(\Pi_\lambda)
\leq
r
\]
for every $\lambda\in\Lambda$. Assume the eigenvalues are known and the entries of every $\Pi_\lambda$, $\lambda\in\Lambda$, are observed on the same Bernoulli-$p$ set $X\subset\mathbb Z_N^2$. If \eqref{eq:graph-sample-bound} holds, then with probability at least
\[
1-|\Lambda|e^{-cpN^2}
\]
there are Fourier-$\ell^1$ reconstructions for which
\begin{equation}\label{eq:heat-recovery-bound}
\|\widetilde H_{t,\Lambda}-e^{-tL}\|_2
\leq
11.47\epsilon
\sum_{\lambda\in\Lambda}
e^{-t\lambda}\sqrt{m(\lambda)}
+
(
\sum_{\lambda\notin\Lambda}
e^{-2t\lambda}m(\lambda)
)^{\frac{1}{2}}.
\end{equation}
\end{corollary}

\begin{proof}
Apply Corollary \ref{cor:projector-recovery} to each $\lambda\in\Lambda$ and use the union bound. Then substitute
\[
\delta_\lambda
=
11.47\epsilon\sqrt{m(\lambda)}
\]
into Proposition \ref{prop:heat-kernel}.
\end{proof}

If $\Lambda$ contains every eigenvalue below a threshold $R$, then the omitted tail is bounded by
\[
\left(\sum_{\lambda\geq R}
e^{-2t\lambda}m(\lambda)
\right)^{\frac{1}{2}}
\leq
\sqrt N e^{-tR}.
\]
This explains why low-frequency projector information is particularly relevant for long-time diffusion. The heat kernel and its low-frequency truncations are standard tools in diffusion geometry and spectral clustering \cite{Chung,CoifmanLafon}.

The individual quantities $\operatorname{FR}_{\min}(\Pi_\lambda)$ do not provide simultaneous data. Outside families such as circulant graphs, the minimizing labeling may depend on $\lambda$. Corollary \ref{cor:heat-recovery} therefore requires bounds for all retained projectors in one common labeling.

For a fixed labeling $\sigma$ and a threshold $r$, one may define
\[
\Lambda_{\sigma,r}
=
\{\lambda:\operatorname{FR}(\Pi_{\lambda,\sigma})\leq r\}
\]
and the corresponding decomposition
\[
L
=
L_{\operatorname{str}}^{\sigma,r}
+
L_{\operatorname{rem}}^{\sigma,r},
\]
where
\[
L_{\operatorname{str}}^{\sigma,r}
=
\sum_{\lambda\in\Lambda_{\sigma,r}}
\lambda\Pi_\lambda
\]
and
\[
L_{\operatorname{rem}}^{\sigma,r}
=
\sum_{\lambda\notin\Lambda_{\sigma,r}}
\lambda\Pi_\lambda.
\]
This is a coordinate-dependent decomposition of the Laplacian into projectors that are Fourier-compressible in the chosen labeling and those that are not. It should not be defined using separate minimizing labelings for separate eigenspaces.

\section{A graph-signal spectral synthesis principle}\label{sec:synthesis}

The preceding sections concern kernels on $\mathbb Z_N^2$. We now consider functions on the vertex set itself. The purpose of this section is to preserve a second theme of the original problem while keeping its hypotheses and conclusions separate from kernel recovery.

Let $G_N$ be a graph on $\mathbb Z_N$, and let
\[
L_N
=
\sum_\lambda\lambda\Pi_{\lambda,N}
\]
be the spectral decomposition of its Laplacian. In this section, the vertex-space norm is normalized:
\begin{equation}\label{eq:normalized-vertex-norm}
\|u\|_{2,N}
=
(
\frac{1}{N}
\sum_{x\in\mathbb Z_N}|u(x)|^2
)^{\frac{1}{2}}.
\end{equation}
Multiplication of the inner product by $\frac{1}{N}$ does not change the orthogonal projectors. For a graph signal $u_N$, define
\[
c_{\lambda,N}
=
\|\Pi_{\lambda,N}u_N\|_{2,N}.
\]
Orthogonality gives
\begin{equation}\label{eq:block-parseval}
\sum_\lambda c_{\lambda,N}^2
=
\|u_N\|_{2,N}^2.
\end{equation}

For $u_N\neq0$, define the graph spectral Fourier ratio
\begin{equation}\label{eq:sfr-definition}
\operatorname{SFR}_N(u_N)
=
\frac{1}{\sqrt N}
\frac{
\sum_\lambda c_{\lambda,N}
}{
\|u_N\|_{2,N}
}.
\end{equation}
This quantity is the normalized quotient of the $\ell^1$ and $\ell^2$ norms of the vector of spectral-block norms. It is not the same as the Fourier ratio of a projector kernel.

If $q_N$ is the number of distinct Laplacian eigenvalues, then Cauchy--Schwarz gives
\begin{equation}\label{eq:sfr-trivial-bound}
\operatorname{SFR}_N(u_N)
\leq
\sqrt{\frac{q_N}{N}}
\leq
1.
\end{equation}
Thus, the estimate with exponent $\kappa=0$ is automatic. A positive value of $\kappa$ records additional concentration among the spectral blocks.

The following statement is a finite graph analogue of a spectral synthesis principle. The argument is elementary once the normalizations are fixed, but it gives a useful quantitative obstruction to simultaneous spatial concentration and spectral regularity. Related questions on manifolds are studied in \cite{IosevichMayeliWyman}.

\begin{theorem}[Asymptotic graph spectral synthesis]\label{thm:graph-spectral-synthesis}
Let $G_N$ be a sequence of graphs on $\mathbb Z_N$. Let $u_N:\mathbb Z_N\to\mathbb C$ be supported on a set $E_N$ satisfying
\begin{equation}\label{eq:synthesis-support}
|E_N|
\leq
C N^\alpha
\end{equation}
for some $0<\alpha<1$. Suppose that, for some $2\leq p<\infty$,
\begin{equation}\label{eq:spectral-lp}
\left(\sum_\lambda c_{\lambda,N}^p
\right)^{\frac{1}{p}}
\leq
C_p
\end{equation}
uniformly in $N$. Suppose also that, for some $0\leq\kappa\leq\frac{\alpha}{2}$,
\begin{equation}\label{eq:sfr-decay}
\operatorname{SFR}_N(u_N)
\leq
C_\kappa N^{-\kappa}.
\end{equation}
If $\kappa<\frac{\alpha}{2}$, assume
\begin{equation}\label{eq:synthesis-p-range}
2\leq p
<
\frac{2(1-2\kappa)}{\alpha-2\kappa}.
\end{equation}
If $\kappa=\frac{\alpha}{2}$, no additional restriction is placed on the finite exponent $p$. Then
\begin{equation}\label{eq:synthesis-l1-conclusion}
\frac{1}{N}
\sum_{x\in\mathbb Z_N}|u_N(x)|
\longrightarrow
0.
\end{equation}
Consequently, for every sequence $\phi_N:\mathbb Z_N\to\mathbb C$ satisfying $\|\phi_N\|_\infty\leq1$,
\begin{equation}\label{eq:synthesis-test-conclusion}
\frac{1}{N}
\sum_{x\in\mathbb Z_N}u_N(x)\phi_N(x)
\longrightarrow
0.
\end{equation}
\end{theorem}

\begin{proof}
Set
\[
\theta
=
\frac{p-2}{2(p-1)}.
\]
Then
\[
\frac{1}{2}
=
\theta+
\frac{1-\theta}{p}.
\]
Interpolation between $\ell^1$ and $\ell^p$ gives
\[
\left(\sum_\lambda c_{\lambda,N}^2
\right)^{\frac{1}{2}}
\leq
\left(\sum_\lambda c_{\lambda,N}
\right)^\theta
\left(
\sum_\lambda c_{\lambda,N}^p
\right)^{\frac{1-\theta}{p}}.
\]
By \eqref{eq:sfr-definition}, \eqref{eq:sfr-decay}, and \eqref{eq:block-parseval},
\[
\sum_\lambda c_{\lambda,N}
\leq
C_\kappa N^{\frac{1}{2}-\kappa}
\|u_N\|_{2,N}.
\]
Using \eqref{eq:spectral-lp}, we obtain
\[
\|u_N\|_{2,N}
\leq
C
N^{(\frac{1}{2}-\kappa)\theta}
\|u_N\|_{2,N}^\theta.
\]
Therefore,
\[
\|u_N\|_{2,N}
\leq
C
N^{(\frac{1}{2}-\kappa)\frac{p-2}{p}}.
\]

Since $u_N$ is supported on $E_N$, Cauchy--Schwarz gives
\[
\frac{1}{N}
\sum_x|u_N(x)|
\leq
\left(\frac{|E_N|}{N}\right)^{\frac{1}{2}}
\|u_N\|_{2,N}.
\]
The right side is bounded by a constant times
\[
N^{
(\frac{1}{2}-\kappa)\frac{p-2}{p}
+
\frac{\alpha-1}{2}
}.
\]
If $\kappa<\frac{\alpha}{2}$, this exponent is negative precisely when
\[
p(\alpha-2\kappa)
<
2(1-2\kappa),
\]
which is \eqref{eq:synthesis-p-range}. If $\kappa=\frac{\alpha}{2}$, the exponent equals
\[
-\frac{1-\alpha}{p},
\]
which is negative for every finite $p$. This proves \eqref{eq:synthesis-l1-conclusion}. The test-function conclusion follows from
\[
|
\frac{1}{N}
\sum_xu_N(x)\phi_N(x)
|
\leq
\frac{1}{N}
\sum_x|u_N(x)|.
\]
\end{proof}

The conclusion in \eqref{eq:synthesis-l1-conclusion} is equivalent to uniform convergence against all test functions bounded by one, because one may choose the pointwise phase of $u_N$ as the test function. It is therefore more precise to describe the result as normalized $\ell^1$ decay rather than only weak convergence.

If $u_N\geq0$ and
\[
\frac{1}{N}
\sum_xu_N(x)
\geq
c>0,
\]
then the conclusion is impossible. Thus, a sequence of nonnegative signals carrying a fixed positive normalized mass cannot satisfy all of the spatial and spectral concentration hypotheses in Theorem \ref{thm:graph-spectral-synthesis}.

\begin{corollary}[Asymptotic uniqueness from incomplete graph data]\label{cor:graph-incomplete-recovery}
Let $M_N\subset\mathbb Z_N$ satisfy
\[
|M_N|
\leq
C N^\alpha
\]
for some $0<\alpha<1$. Let $f_N,g_N:\mathbb Z_N\to\mathbb C$ agree on $\mathbb Z_N\setminus M_N$. Suppose that, for some
\[
2\leq p<\frac{2}{\alpha},
\]
both sequences satisfy
\[
\left(\sum_\lambda
\|\Pi_{\lambda,N}f_N\|_{2,N}^p
\right)^{\frac{1}{p}}
\leq
C_p
\]
and
\[
\left(\sum_\lambda
\|\Pi_{\lambda,N}g_N\|_{2,N}^p
\right)^{\frac{1}{p}}
\leq
C_p.
\]
Then
\begin{equation}\label{eq:incomplete-l1}
\frac{1}{N}
\sum_{x\in\mathbb Z_N}|f_N(x)-g_N(x)|
\longrightarrow
0.
\end{equation}
\end{corollary}

\begin{proof}
Let $h_N=f_N-g_N$. Then $h_N$ is supported on $M_N$, and the triangle inequality gives a uniform $\ell^p$ bound for its spectral-block norms. If $h_N\neq0$, the trivial estimate \eqref{eq:sfr-trivial-bound} gives
\[
\operatorname{SFR}_N(h_N)
\leq
1.
\]
Theorem \ref{thm:graph-spectral-synthesis} applies with $\kappa=0$ and gives \eqref{eq:incomplete-l1}. If $h_N=0$, the conclusion is immediate.
\end{proof}

Corollary \ref{cor:graph-incomplete-recovery} is an asymptotic uniqueness statement in normalized $\ell^1$. It does not say that $f_N=g_N$ for each fixed $N$. The conclusion is that two spectrally regular signals agreeing outside a sublinear exceptional set cannot remain macroscopically different on that exceptional set.

There is an exact finite-dimensional continuation mechanism when the admissible signal lies in one eigenspace. Let $M\subset\mathbb Z_N$ be the missing set and let $\Omega=\mathbb Z_N\setminus M$. Decompose the Laplacian into blocks relative to $M\cup\Omega$.

\begin{proposition}[Continuation of a Laplacian eigenfunction]\label{prop:helmholtz-continuation}
Suppose
\[
(L-\lambda I)u=0
\]
and the values of $u$ are known on $\Omega$. If the principal matrix
\[
L_{MM}-\lambda I
\]
is invertible, then the values on $M$ are uniquely determined by
\begin{equation}\label{eq:helmholtz-formula}
u_M
=
-(L_{MM}-\lambda I)^{-1}L_{M\Omega}u_\Omega.
\end{equation}
\end{proposition}

\begin{proof}
The rows of $(L-\lambda I)u=0$ indexed by $M$ give
\[
(L_{MM}-\lambda I)u_M
=
-L_{M\Omega}u_\Omega.
\]
Invertibility gives \eqref{eq:helmholtz-formula}.
\end{proof}

This is a discrete Helmholtz-type boundary value problem. If the principal matrix is singular, nontrivial eigenfunctions may vanish on $\Omega$, and uniqueness can fail. The exact continuation statement and the asymptotic synthesis theorem address different regimes, but both express the principle that spectral restrictions limit the freedom to modify a signal on a missing set.

\section{Open problems}\label{sec:questions}

The results above identify a structural invariant and its consequences once a useful labeling is known. They also separate several questions that were conflated in the original formulation. The optimization over labelings, stable approximation of a displayed kernel, exact recovery of a graph, and recovery of spectral projectors from adjacency data are different problems.

\begin{question}[Finding a good labeling]
What is the computational complexity of evaluating or approximating $\operatorname{FR}_{\min}(G)$? Can one construct a labeling whose Fourier ratio is provably within a controlled factor of the minimum for natural graph classes?
\end{question}

This is a Fourier-analytic graph layout problem. Classical layout objectives favor local geometric organization, while the Fourier ratio favors global additive organization. Spectral ordering, seriation, and multiscale partitioning are natural heuristics, but there is presently no general approximation theorem for the Fourier-ratio objective.

\begin{question}[Graph-specific inverse structure]
Corollary \ref{cor:coset-structure} gives an exact signed coset decomposition for every Boolean kernel of bounded Fourier ratio. What additional conclusions follow when the kernel is symmetric, has zero diagonal, and is the adjacency kernel of a simple graph? Can the cancellation in the signed decomposition be controlled, and can the cosets be chosen in symmetric pairs or as graphs of affine and congruence relations? Can the general bound $\Phi(M)=\exp(M^{4+o(1)})$ be improved for adjacency kernels, regular graphs, or sparse graphs?
\end{question}

The general Fourier algebra theorem settles the existence of a dimension-free additive decomposition. The remaining problem is to use graph constraints to make that decomposition more rigid, more economical, and more directly interpretable as an edge relation.

\begin{question}[Random labelings and random graphs]
What is the typical Fourier ratio of a fixed graph under a uniformly random relabeling? What is the typical order of $\operatorname{FR}_{\min}(G)$ for an Erd\H{o}s--R\'enyi graph or a random regular graph \cite{Bollobas}?
\end{question}

The random-phase heuristic predicts a value of order $N$ for a generic labeling. The energy lower bound for many dense random graphs is only of order $\sqrt N$. Determining how much the minimization over $N!$ labelings can reduce the generic value appears to be a difficult question.

\begin{question}[Exact and constrained graph recovery]
Can one prove recovery theorems that incorporate symmetry, the zero diagonal condition, and the binary constraints on adjacency entries, and that yield Hamming-error or exact edge-recovery guarantees under a Fourier-ratio hypothesis?
\end{question}

The thresholding estimate in Remark \ref{rem:recovery-scope} is only a consequence of Frobenius approximation. A theorem designed for the discrete graph class may improve both the formulation and the required accuracy.

\begin{question}[From adjacency samples to spectral projectors]
Under what spectral-gap, perturbation, or structural hypotheses can incomplete adjacency observations be converted into accurate estimates of low-frequency projector kernels? Can such a result be combined with the Fourier-ratio bounds of Section \ref{sec:projectors} without assuming direct access to projector entries?
\end{question}

This problem is necessary if projector recovery is to be derived from ordinary graph observations. It would require control of the reconstructed adjacency operator, its spectrum, and the stability of its spectral subspaces.

\begin{question}[Common harmonic coordinates]
For which graph families is there one labeling that simultaneously minimizes, or approximately minimizes, the Fourier ratios of the adjacency matrix and a prescribed collection of Laplacian projectors?
\end{question}

Circulant graphs have exact simultaneous minimization. The Petersen graph shows that adjacency minimizers and projector minimizers can be disjoint. It would be useful to identify structural conditions lying between these two cases.

\begin{question}[Relations with graph layout and graph width]
Can $\operatorname{FR}_{\min}(G)$ be bounded in terms of bandwidth, minimum linear arrangement, treewidth, expansion, or another classical graph parameter? Are there graph families for which these quantities and the Fourier-ratio layout objective provably disagree?
\end{question}

The interval and parity labelings of the balanced complete bipartite graph already show that additive organization is not the same as interval locality. More systematic separations would clarify the geometry detected by the Fourier ratio.

\begin{question}[Perturbation and robust coordinates]
Suppose two graphs differ by a small number of edge additions or deletions. Under what conditions do they admit one labeling in which both Fourier ratios are small? Can Proposition \ref{prop:fr-perturbation} be improved by using the structure of graph edits rather than the general bound $\|\widehat H\|_1\leq N\|H\|_2$?
\end{question}

\begin{question}[Other finite abelian groups and infinite limits]
How should the invariant be formulated when the vertex set carries a noncyclic finite abelian group structure? For Cayley graphs of $\mathbb Z^d$, can one define a meaningful finite-volume Fourier ratio on growing boxes and relate its limiting behavior to the convolution symbol, growth, or isoperimetry?
\end{question}

A direct $\ell^1$ and $\ell^2$ ratio for an infinite translation-invariant adjacency kernel is usually infinite. A finite-volume or symbol-based formulation is therefore required.

\begin{question}[Sharp spectral synthesis]
Are the exponent range and spectral-block hypotheses in Theorem \ref{thm:graph-spectral-synthesis} sharp? What graph geometry forces nontrivial decay of $\operatorname{SFR}_N(u_N)$, and when can normalized $\ell^1$ uniqueness be strengthened to exact uniqueness or a quantitative stability estimate?
\end{question}

These questions suggest that Fourier ratio methods may connect graph layout, compressed sensing, spectral geometry, and inverse problems. The main results of this paper give exact lower bounds and exact model families. The larger structural and algorithmic theory remains open.

\appendix

\section{Reproducibility of the Petersen computation}\label{app:petersen-code}

The following program reproduces Proposition \ref{prop:petersen-computation}. It constructs the Petersen graph in its Kneser representation, constructs its automorphism group from the adjacent transpositions of five points, and evaluates one representative from each of the $30240$ left cosets in $S_{10}$. The program was tested with Python 3.11, NumPy 2.3.5, and SymPy 1.14.0. No external graph data are used.

\begin{verbatim}
import itertools
import numpy as np
from sympy.combinatorics import Permutation, PermutationGroup
from sympy.combinatorics.named_groups import SymmetricGroup

pairs = list(itertools.combinations(range(5), 2))
index = {p: i for i, p in enumerate(pairs)}

def induced_perm(images):
    out = []
    for a, b in pairs:
        q = tuple(sorted((images[a], images[b])))
        out.append(index[q])
    return Permutation(out)

gens = []
for i in range(4):
    images = list(range(5))
    images[i], images[i + 1] = images[i + 1], images[i]
    gens.append(induced_perm(images))

aut = PermutationGroup(gens)
sym = SymmetricGroup(10)
# SymPy returns right-coset representatives. Inversion gives
# left-coset representatives for the action on vertex values.
reps = [p ** -1 for p in sym.coset_transversal(aut)]

A = np.zeros((10, 10), dtype=np.longdouble)
for i, p in enumerate(pairs):
    for j, q in enumerate(pairs):
        if set(p).isdisjoint(q):
            A[i, j] = 1

I = np.eye(10, dtype=np.longdouble)
J = np.ones((10, 10), dtype=np.longdouble)
P2 = (A + 2 * I - J / 2) / 3
P5 = (-A + I + J / 5) / 3

N = 10
U = np.empty((N, N), dtype=np.clongdouble)
for m in range(N):
    for x in range(N):
        z = np.clongdouble(-2j * np.pi * m * x / N)
        U[m, x] = np.exp(z) / np.sqrt(np.longdouble(N))

def fourier_ratio(M, permutation):
    order = np.array([permutation(i) for i in range(N)])
    Ms = M[np.ix_(order, order)].astype(np.clongdouble)
    transformed = U @ Ms @ U
    numerator = np.sum(
        np.abs(transformed), dtype=np.longdouble
    )
    denominator = np.sqrt(
        np.sum(np.abs(Ms) ** 2, dtype=np.longdouble)
    )
    return float(numerator / denominator)

matrices = [A, P2, P5]
values = [[] for M in matrices]
mins = [float("inf") for M in matrices]
minimizers = [[] for M in matrices]
tolerance = 1e-12

for permutation in reps:
    for j, M in enumerate(matrices):
        value = fourier_ratio(M, permutation)
        values[j].append(value)
        if value < mins[j] - tolerance:
            mins[j] = value
            minimizers[j] = [permutation]
        elif abs(value - mins[j]) <= tolerance:
            minimizers[j].append(permutation)

sets = []
for group in minimizers:
    sets.append({
        tuple(p(i) for i in range(N)) for p in group
    })

print("automorphism order", aut.order())
print("orbit representatives", len(reps))
print("minima", mins)
print("representative counts", [len(x) for x in minimizers])
print("labeling counts", [120 * len(x) for x in minimizers])
print("intersections", len(sets[0] & sets[1]),
      len(sets[0] & sets[2]), len(sets[1] & sets[2]))

for j, row in enumerate(values):
    distinct = []
    for value in sorted(row):
        if not distinct or value - distinct[-1] > 1e-10:
            distinct.append(value)
    print("next gap", j, distinct[1] - distinct[0])
\end{verbatim}

The output gives the minima stated in Proposition \ref{prop:petersen-computation}, representative counts $8,2,2$, labeling counts $960,240,240$, intersection counts $0,0,2$, and next-value gaps exceeding $0.143$, $0.295$, and $0.329$.

\section*{Funding}

The work of Alex Iosevich was supported in part by the National Science Foundation under grant DMS-2506858. The funder had no role in the design of the study, the mathematical analysis, the preparation of the manuscript, or the decision to submit the article for publication.

\section*{Data availability}

All data needed to verify the theoretical results are contained in the manuscript. The Petersen computation uses no external dataset, and complete source code is included in Appendix \ref{app:petersen-code}.

\section*{Declaration of competing interest}

The authors declare that they have no known competing financial interests or personal relationships that could have appeared to influence the work reported in this article.

\section*{Declaration of generative AI and AI-assisted technologies in the manuscript preparation process}

During the preparation of this work, the authors used OpenAI ChatGPT to assist with literature discovery, structural revision, mathematical checking, and LaTeX preparation. After using this tool, the authors reviewed and edited the content as needed and take full responsibility for the content of the article.

\end{document}